\documentclass[english,12pt,oneside]{amsproc}
\usepackage[T2A]{fontenc}
\usepackage[english]{babel}
\usepackage{sseq}
\usepackage{graphics}
\usepackage{amsfonts, amssymb, amscd, amsmath}
\usepackage{latexsym}
\usepackage[matrix,arrow,curve]{xy}
\usepackage{mathabx,mathtools}
\usepackage{color}
\usepackage{mathrsfs}
\usepackage{mathdots}
\usepackage{pigpen}
\usepackage{tikz}
\usetikzlibrary{matrix}

\DeclareMathOperator{\pt}{pt} 
 \DeclareMathOperator{\id}{id}

\DeclareMathOperator{\sgn}{sgn}
 \DeclareMathOperator{\im}{Im}
 
\DeclareMathOperator{\ver}{Vert} 
 \DeclareMathOperator{\fac}{Fac}
\DeclareMathOperator{\spi}{Sp} \DeclareMathOperator{\dia}{Df}
\DeclareMathOperator{\soc}{Soc}

\newcommand{\Zo}{\mathbb{Z}}

\newcommand{\Co}{\mathbb{C}}

\newcommand{\ko}{\Bbbk}

\newcommand{\simc}{\!\!\sim}

\newcommand{\br}{\widetilde{\beta}}

\newcommand{\ddb}{b_1}
\newcommand{\capp}{\!\!\frown\!\!}

\newcommand{\eqd}{\stackrel{\text{\tiny def}}{=}}

\newcommand{\minel}{\hat{0}}

\newcommand{\less}[1]{\stackrel{#1}{<}}
\newcommand{\gess}[1]{\stackrel{#1}{>}}

\newcommand{\wh}[1]{{\widehat{#1}}}

\newcommand{\inc}[2]{{[#1\!:\!#2]}}

\newcommand{\E}[1]{(E_{#1})}
\newcommand{\dif}[1]{(d_{#1})}

\newcommand{\Ea}[1]{(\dot{E}_{#1})}
\newcommand{\difa}[1]{(\dot{d}_{#1})}
\newcommand{\fa}{\dot{f}}

\newcommand{\difm}[1]{d_{#1}^{-}}

\newcommand{\Hr}{\widetilde{H}}
\newcommand{\dd}{\partial}

\newcounter{stmcounter}[section]
\newcounter{thcounter}

\numberwithin{equation}{section}

\theoremstyle{plain}
\newtheorem{cor}[stmcounter]{Corollary}

\newtheorem{thm}[thcounter]{Theorem}
\newtheorem{prop}[stmcounter]{Proposition}
\newtheorem{lemma}[stmcounter]{Lemma}
\newtheorem{defin}[stmcounter]{Definition}

\theoremstyle{definition}

\newtheorem{rem}[stmcounter]{Remark}
\newtheorem{con}[stmcounter]{Construction}

\begin{document}

\title[Homology cycles in torus manifolds]{Homology cycles in manifolds with locally standard torus actions}

\author[Anton Ayzenberg]{Anton Ayzenberg}
\address{Department of Mathematics, Osaka City University, Sumiyoshi-ku, Osaka 558-8585, Japan.}
\email{ayzenberga@gmail.com}

\date{\today}
\thanks{The author is supported by the JSPS postdoctoral fellowship program.}

\begin{abstract}
Let $X$ be a $2n$-manifold with a locally standard action of a
compact torus $T^n$. If the free part of action is trivial and
proper faces of the orbit space $Q$ are acyclic, then there are
three types of homology classes in $X$: (1) classes of face
submanifolds; (2) $k$-dimensional classes of $Q$ swept by actions
of subtori of dimensions $<k$; (3) relative $k$-classes of $Q$
modulo $\partial Q$ swept by actions of subtori of dimensions
$\geqslant k$. The submodule of $H_*(X)$ spanned by face classes
is an ideal in $H_*(X)$ with respect to the intersection product.
It is isomorphic to $(\mathbb{Z}[S_Q]/\Theta)/W$, where
$\mathbb{Z}[S_Q]$ is the face ring of the Buchsbaum simplicial
poset $S_Q$ dual to $Q$; $\Theta$ is the linear system of
parameters determined by the characteristic function; and $W$ is a
certain submodule, lying in the socle of $\mathbb{Z}[S_Q]/\Theta$.
Intersections of homology classes different from face submanifolds
are described in terms of intersections on $Q$ and~$T^n$.
\end{abstract}

\maketitle


%
%
%
%
%
%
%

\section{Introduction}\label{SecIntro}

An action of a compact torus $T^n$ on a smooth compact manifold
$M$ of dimension $2n$ is called locally standard if it is locally
isomorphic to the standard action of $T^n$ on~$\Co^n$. The orbit
space $Q=M/T^n$ has a natural structure of a manifold with corners
in which open $k$-dimensional faces of $Q$ correspond to
$k$-dimensional orbits of an action. Every manifold with locally
standard torus action is equivariantly homeomorphic to the
quotient model $X=Y/\simc$, where $Y$ is a principal $T^n$-bundle
over $Q$ and $\sim$ is an equivalence relation determined by the
characteristic function on $Q$ \cite{Yo}.

This paper is the third in a series of works, where we study
topology of $X$ under the assumption that proper faces of the
orbit space are acyclic and $Y$ is a trivial bundle. Previous
works \cite{AyV1,AyV2} were devoted to the homological spectral
sequence associated with the filtration of $X$ by orbit types. In
this paper we give a geometrical description of homology cycles on
$X$.

In the case when all faces of $Q$ including $Q$ itself are
acyclic, the topology of the corresponding manifold $X$ is known
(see \cite{MasPan}). In this case the equivariant cohomology ring
is isomorphic to the face ring of the simplicial poset $S_Q$ dual
to $Q$: $H^*_T(X;\Zo)\cong\Zo[S_Q]$. As a ring, it is generated by
equivariant cycles, dual to face submanifolds of $X$ (these
generators correspond to the standard generators of the face
ring). The spectral sequence of the Borel fibration
$ET^n\times_TX\to BT^n$ collapses at a second page. A
fiber-inclusion map $\iota\colon X\hookrightarrow ET^n\times_TX$
induces a surjective ring homomorphism $\iota^*\colon
H^*_T(X;\Zo)\to H^*(X;\Zo)$, whose kernel is the image of
$H^{>0}(BT^n;\Zo)$ under $\pi^*$. Thus, $H^*(X;\Zo)\cong
\Zo[S_Q]/(\theta_1,\ldots,\theta_n)$, where $\theta_i$ are the
images of generators $v_i$ of the ring
$H^*(BT^n;\Zo)\cong\Zo[v_1,\ldots,v_n]$. The sequence
$(\theta_1,\ldots,\theta_n)$ is a linear system of parameters in
$\Zo[S_Q]$. Since $S_Q$ is Cohen--Macaulay, this is a regular
sequence and $\dim H^{2k}(X)=h_k(S_Q)$.

In the case when only \emph{proper} faces of $Q$ are acyclic, this
approach is inapplicable. The spectral sequence of the Borel
fibration does not collapse at a second page. We still have the
ring homomorphism $H^*_T(X)/(\pi^*H^{>0}(BT^n))\to H^*(X)$, but it
is neither injective nor surjective.

Nevertheless, there is an apparent connection between topology of
spaces with torus actions and the theory of face rings. In
\cite{AMPZ} we proved that there exists an isomorphism of rings
(and $H^*(BT^n)$-modules):
\begin{equation}\label{eqEqCohomOur}
H^*_T(X;\Zo)\cong \Zo[S_Q]\oplus H^*(Q;\Zo)
\end{equation}
(the units of the rings in the direct sum are identified).

When proper faces of $Q$ are acyclic, the dual simplicial poset
$S_Q$ is Buchsbaum \cite[Cor.6.3]{AyV1}. There is a standard tool
in combinatorics and commutative algebra devised to study
Buchsbaum simplicial complexes, namely, the $h'$-vector. By
definition, the $h'$-numbers of a Buchsbaum simplicial poset $S$
are the dimensions of homogeneous components of the quotient
algebra $\ko[S]/(\theta_1,\ldots,\theta_n)$, where
$\theta_1,\ldots,\theta_n$ is any linear system of parameters.
These numbers do not depend on a linear system of parameters, and
can be expressed in terms of the ordinary $h$-numbers and Betti
numbers of $S$ (see \cite{Sch,NS} or Definition
\ref{definHvectors} below). In \cite[Th.3]{AyV2} we proved that
$\dim \E{X}^2_{q,q}=h'_{n-q}(S_Q)$, where $\E{X}^*_{*,*}$ is the
homological spectral sequence associated with the orbit type
filtration of $X$.

In this paper we describe the geometrical structure of homology
cycles on $X$.

\begin{thm}
Homology classes of $X$ have three different types:
\begin{enumerate}
\item the classes of face submanifolds (we call them face classes);

\item the classes, represented by $k$-cycles of $Q$, swept
by an action of a subtorus of dimension $<k$ (these classes will
be called spine classes);

\item the classes, represented by relative $k$-cycles of $Q$ modulo $\dd
Q$ with $k<n$, swept by an action of a subtorus of dimension
$\geqslant k$ (these classes will be called diaphragm classes).
\end{enumerate}

Linear relations on face classes are of two types: the relations
appearing in the ring $\ko[S_Q]/(\theta_1,\ldots,\theta_n)$, and
additional relations lying in a socle of
$\ko[S_Q]/(\theta_1,\ldots,\theta_n)$.

Intersections of face classes are encoded by the multiplication in
the face ring of~$S_Q$. Proper face classes span the ideal of
$H_*(X)$ with respect to the intersection product. Intersections
of other classes are described by means of the intersection
products on $Q$ and $T^n$.
\end{thm}

Precise statements are given in Propositions \ref{propHXstruct},
\ref{propTwoTypesOfRels}, \ref{propHXstructIntegers},
\ref{propIntersectFace}, \ref{propIntersections},
\ref{propFaceIsIdeal}, and Theorem \ref{thmSocle}.

Face classes and the elements of $H_k(Q,\dd Q)$ swept by the
action of the whole group $T^n$ are equivariant. This gives an
independent geometrical evidence for the formula
\eqref{eqEqCohomOur}.


The paper may be briefly outlined as follows. Section
\ref{SecPrelim} contains basic definitions and outlines the
previous results. In Section \ref{SecComput} we make technical
preparations for Section \ref{SecFaceSubmfds}, which is devoted to
linear relations on face classes. In Section \ref{SecOtherCycles}
we realize non-face classes of $X$ as embedded pseudomanifolds.
These geometrical constructions imply a partial description of
intersection theory on $X$, which is done in Section
\ref{SecIntersect}.

Two examples of computations are discussed in Section
\ref{SecExamples}. A very particular $4$-dimensional example is
worked out, and the reader is encouraged to refer to it while
reading other parts of the paper. The second example is more
general: we apply our technique to the class of orientable toric
origami manifolds with acyclic proper faces of the orbit space,
and rediscover some results of \cite{AMPZ}. A supplementary space
$\wh{X}$ is introduced in the last section. This space can be
considered as a $T^n$-invariant tubular neighborhood of the union
of characteristic submanifolds in $X$. By using intersection
theory on $\wh{X}$, we prove that certain elements of
$\ko[S_Q]/(\theta_1,\ldots,\theta_n)$ lie in the socle of this
module. This gives a geometrical interpretation of the result
obtained by Novik--Swartz~\cite{NS}.

%
%
%
%
%
%
%

\section{Preliminaries and previous results}\label{SecPrelim}

%
%
%

\subsection{Manifolds with locally standard torus actions}

An action of $T^n$ on a (compact connected smooth) manifold
$M^{2n}$ is called \emph{locally standard}, if $M$ has an atlas of
$T^n$-invariant charts, each equivalent to an open $T^n$-invariant
subset of the standard action of $T^n$ on $\Co^n$. The reader is
referred to \cite{BPnew} or \cite{Yo} for the precise definition.
The orbit space of a locally standard action is a compact
connected $n$-dimensional manifold with corners with the property
that every codimension $k$ face of $Q$ lies in exactly $k$ facets
of $Q$ (such manifolds with corners were called \emph{nice} in
\cite{MasPan}, or \emph{manifolds with faces} elsewhere).

\begin{defin}
A finite partially ordered set (poset) $S$ is called simplicial if
(1)~there is a minimal element $\minel\in S$; (2)~for each element
$J\in S$ the lower order ideal $\{I\in S\mid I\leqslant J\}$ is
isomorphic to the poset of faces of a $k$-simplex for some number
$k$, called the dimension of $I$.
\end{defin}

The elements of $S$ are called simplices. Simplices of dimension
$0$ are called vertices. The number $|I|=\dim I+1$ is equal to the
number of vertices of $I$ and is called the rank of $I$. The set
of vertices of a simplicial poset or a simplex is denoted by
$\ver(\cdot)$.

Every manifold with corners $Q$ determines a dual poset $S_Q$
whose elements are the faces of $Q$ ordered by the reversed
inclusions. When $Q$ is a nice connected manifold with corners,
$S_Q$ is a simplicial poset. We denote abstract elements of $S_Q$
by $I$, $J$, etc. and the corresponding faces of $Q$ by $F_I$,
$F_J$, etc. There holds $\dim F_I=n-|I|$. The minimal element of
$S_Q$ corresponds to the maximal face of $Q$, i.e. $Q$ itself.
Vertices of $S_Q$ correspond to facets of $Q$. The set of facets
of $Q$ is denoted by $\fac(Q)$.

Let $Q$ be the orbit space of locally standard action, and let
$x\in F^\circ$ be a point in the interior of a facet $F\in
\fac(Q)$. Then the stabilizer of $x$, denoted by $\lambda(F)$, is
a 1-dimensional toric subgroup in $T^n$. If $F_I$ is a codimension
$k$ face of $Q$, contained in facets $F_1,\ldots,F_k\in \fac(Q)$,
then the stabilizer of $x\in F_I^{\circ}$ is the $k$-dimensional
torus $T_I=\lambda(F_1)\times\ldots\times\lambda(F_k)\subset T^n$,
where the product is free inside $T^n$. This puts a specific
restriction on subgroups $\lambda(F)$. In general, a map
\begin{equation}\label{eqCharFunc}
\lambda\colon \fac(Q)\to \{\mbox{1-dimensional toric subgroups of
} T^n\}
\end{equation}
is called \emph{a characteristic function}, if, whenever facets
$F_1,\ldots,F_k$ have nonempty intersection, the map
\[
\lambda(F_1)\times\ldots\times\lambda(F_k)\to T^n,
\]
induced by inclusions $\lambda(F_i)\hookrightarrow T^n$, is
injective and splits. This condition is called $(\ast)$-condition.
Let $i\in \ver(S_Q)$ be the vertex of $S_Q$, and
$T_i=\lambda(F_i)$ be the value of characteristic function. Let
$\omega_i\in H_1(T^n;\ko)\cong\ko^n$ be the fundamental class of
$T_i$. This class is defined uniquely up to sign.

Let $\mu\colon M\to Q$ be the projection to the orbit space. The
free part of the action has the form $\mu|_{Q^{\circ}}\colon
\mu^{-1}(Q^{\circ})\to Q^{\circ}$, where $Q^{\circ}=Q\setminus \dd
Q$ is the interior of the manifold with corners. It is a principal
torus bundle over $Q^{\circ}$ which can be uniquely extended over
$Q$; it determines a principal $T^n$-bundle $\rho\colon Y\to Q$.
Therefore any manifold with locally standard action defines three
objects: a nice manifold with corners $Q$, a principal torus
bundle $\rho\colon Y\to Q$, and a characteristic
function~$\lambda$. One can recover the manifold $M$, up to
equivariant homeomorphism, from these data by the following
standard construction.

\begin{con}[Quotient construction]\label{conModelSpace}
Let $\rho\colon Y\to Q$ be a principal $T^n$-bundle over a nice
manifold with corners, and $\lambda$ be a characteristic function
on $\fac(Q)$. Consider the space $X\eqd Y/\sim$, where $y_1\sim
y_2$ if and only if $\rho(y_1)=\rho(y_2)\in F_I^{\circ}$ for some
face $F_I$ of $Q$, and $y_1,y_2$ lie in the same $T_I$-orbit of
the $T^n$-action on $Y$. Let $f\colon Y\to X$ be the quotient map.
\end{con}

Every manifold $M$ with locally standard torus action is
equivariantly homeomorphic to its model $X$ (\cite[Cor.2]{Yo}). In
the rest of the paper we use the model $X$ instead of $M$.

\begin{rem}
In the paper we work with a smooth manifold with corners $Q$ and
smooth manifolds $X\cong M$, but this is done basically to
simplify the exposition. The quotient model $X=(Q\times
T^n)/\simc$ can obviously be defined for a larger class of spaces.
If $Q$ is a homology manifold with a simple stratification of the
boundary, in which faces are homology manifolds with boundaries,
then $X$ is a closed homology manifold. All results of the paper
are valid in this setting as can be seen from the proofs.
\end{rem}

%
%
%

\subsection{Filtrations}

There are natural topological filtrations on $Q$, $Y$ and $X$.
Namely, $Q_k\subseteq Q$ is the union of $k$-dimensional faces of
$Q$, $Y_k=\rho^{-1}(Q_k)\subseteq Y$, and $X_k=f(Y_k)\subset X$ is
the union of toric orbits of dimension at most $k$. The maps
$\mu\colon X\to Q$, $\rho\colon Y\to Q$ and $f\colon Y\to X$
respect these filtrations. The homological spectral sequences
produced by these filtrations are denoted $\E{Q}^*_{*,*}$,
$\E{Y}^*_{*,*}$, and $\E{X}^*_{*,*}$. The map $f$ induces the
morphism of spectral sequences $f^r_*\colon\E{Y}^r\to\E{X}^r$.

The subsets $\rho^{-1}(F_I)\subset Y$ and $\mu^{-1}(F_I) \subset
X$ which cover the face $F_I\subset Q$ are denoted $Y_I$ and $X_I$
respectively. Note that the subset $X_I$ is a closed submanifold
of $X$ of codimension $2|I|$. It is called a \emph{face
submanifold}. Face submanifolds of codimension $2$ are called
\emph{characteristic submanifolds}. They correspond to facets of
$Q$.

The first page of $\E{Q}^*_{*,*}$ has the form
\[
\E{Q}^1_{p,q} = H_{p+q}(Q_p,Q_{p-1})\cong \bigoplus_{I\in S_Q,
\dim F_I=p}H_{p+q}(F_I,\dd F_I).
\]
and the first differential $\dif{Q}^1$ is the sum of the maps
\begin{multline}\label{eqMattachingMap}
m^q_{I,J}\colon H_{q+\dim F_I}(F_I,\dd F_I)\to H_{q+\dim
F_I-1}(\dd F_I)\to \\ \to H_{q+\dim F_I-1}(\dd F_I, \dd
F_I\setminus F_J^{\circ})\cong H_{q+\dim F_J}(F_J, \dd F_J),
\end{multline}
defined for every face $F_I$ and $F_J\in \fac(F_I)$. Here the
first map is the connecting homomorphism in the homology exact
sequence of $(F_I,\dd F_I)$, and the last isomorphism is due to
excision.

%
%
%

\subsection{Almost acyclic case}

Let $\ko$ be a ground ring. When coefficients in the notation of
(co)homology are omitted, they are supposed to be in $\ko$. From
now on we impose two restrictions on $X$ mentioned in the
introduction. First, $Q$ is an orientable manifold and all its
proper faces are acyclic (over $\ko$). Second, the principal torus
bundle $Y\to Q$ is trivial. Thus $X=(Q\times T^n)/\simc$. The
following propositions were proved in \cite{AyV1}, \cite{AyV2}.

\begin{prop}\label{propSisBuch}
The poset $S_Q$ is a Buchsbaum simplicial poset (over $\ko$).
\end{prop}

\begin{prop}\label{propEQstruct}
There exists a homological spectral sequence
$\Ea{Q}^r_{p,q}\Rightarrow H_{p+q}(Q)$, $\difa{Q}^r\colon
\Ea{Q}^r_{p,q}\to \Ea{Q}^r_{p-r,q+r-1}$ with the properties:
\begin{enumerate}
\item $\Ea{Q}^1=H(\E{Q}^1, \difm{Q})$, where the
differential $\difm{Q}\colon \E{Q}^1_{p,q}\to \E{Q}^1_{p-1,q}$
coincides with $\dif{Q}^1$ for $p<n$, and vanishes otherwise.

\item The module $\Ea{Q}^r_{*,*}$ coincides with $\E{Q}^r_{*,*}$
for $r\geqslant 2$ .

\item
\[
\Ea{Q}^1_{p,q}=\begin{cases}
H_p(\dd Q),\mbox{ if } q=0, p<n;\\
H_{q+n}(Q,\dd Q),\mbox{ if } p=n, q\leqslant 0;\\
0,\mbox {otherwise.}
\end{cases}
\]

\item Nontrivial differentials for $r\geqslant 1$ have pairwise different domains
and targets. They have the form $\difa{Q}^r\colon
\Ea{Q}^r_{n,1-r}\to \Ea{Q}^r_{n-r,0}$ and coincide with the
connecting homomorphisms $\delta_{n+1-r}\colon H_{n+1-r}(Q,\dd
Q)\to H_{n-r}(\dd Q)$.
\end{enumerate}
\end{prop}

Let $\Lambda_*$ denote the homology module of a torus:
$\Lambda_*=\bigoplus_{s}\Lambda_s$, $\Lambda_s=H_s(T^n)$.

\begin{prop}\label{propEYstruct}
There exists a homological spectral sequence
$\Ea{Y}^r_{p,q}\Rightarrow H_{p+q}(Y)$ such that
\begin{enumerate}
\item $\Ea{Y}^1=H(\E{Y}^1, \difm{Y})$, where the differential
$\difm{Y}\colon \E{Y}^1_{p,q}\to \E{Y}^1_{p-1,q}$ coincides with
$\dif{Q}^1$ for $p<n$, and vanishes otherwise.
\item $\Ea{Y}^r=\E{Y}^r$ for $r\geqslant 2$.
\item $\Ea{Y}^r_{p,q} = \bigoplus_{q_1+q_2=q}\Ea{Q}^r_{p,q_1}\otimes
\Lambda_{q_2}$ and $\difa{Y}^r=\difa{Q}^r\otimes\id_{\Lambda}$ for
$r\geqslant 1$.
\end{enumerate}
\end{prop}

\begin{prop}\label{propEXstruct}
There exists a homological spectral sequence
$\Ea{X}^r_{p,q}\Rightarrow H_{p+q}(X)$ and the morphism of
spectral sequences $\fa_*^r\colon\Ea{Y}^r_{*,*}\to \Ea{X}^r_{*,*}$
such that:
\begin{enumerate}
\item $\Ea{X}^1=H(\E{X}^1, \difm{X})$ where the differential
$\difm{X}\colon \E{X}^1_{p,q}\to \E{X}^1_{p-1,q}$ coincides with
$\dif{X}^1$ for $p<n$, and vanishes otherwise. The map $\fa_*^1$
is induced by $f_*^1\colon \E{Y}^1\to \E{X}^1$.
\item $\Ea{X}^r=\E{X}^r$ for $r\geqslant 2$.
\item $\E{X}^1_{p,q}=\Ea{X}^1_{p,q}=0$ for $p<q$.
\item $\fa_*^1\colon\Ea{Y}^1_{p,q}\to \Ea{X}^1_{p,q}$ is an
isomorphism for $p>q$ and injective for $p=q$.
\item As a consequence of previous items, for $r\geqslant 1$,
the differentials $\difa{X}^r$ are either isomorphic to
$\difa{Y}^r$ (when they hit the cells with $p>q$), or isomorphic
to the composition of $\difa{Y}^r$ with $\fa_*^r$ (when they hit
the cells with $p=q$), or zero (otherwise).
\item The ranks of diagonal terms at a second page are the $h'$-numbers
of the poset $S_Q$ dual to the orbit space: $\dim
\Ea{X}^2_{q,q}=\dim \E{X}^2_{q,q}=h'_{n-q}(S_Q)$.
\end{enumerate}
\end{prop}

Recall the definition of $h'$-numbers.

\begin{defin}\label{definHvectors}
Let $S$ be a pure simplicial poset, $\dim S=n-1$. Let $f_k$ be the
the number of $k$-dimensional simplices in $S$. The array
$(f_{-1}=0,f_0,\ldots,f_{n-1})$ is called the $f$-vector of $S$.
Define $h$-numbers by the relation:
\[
h_0s^n+h_1s^{n-1}+\ldots+h_n =
f_{-1}(s-1)^n+f_0(s-1)^{n-1}+\ldots+f_{n-1}.
\]
Let $\br_k(S)=\dim \Hr_k(S)$. Define $h'$-numbers by the relation
\[
h_k'=h_k+{n\choose
k}\left(\sum_{j=1}^{k-1}(-1)^{k-j-1}\br_{j-1}(S)\right)\mbox{ for
} 0\leqslant k\leqslant n.
\]
\end{defin}

Propositions \ref{propEQstruct}--\ref{propEXstruct} yield the
description of $H_*(X)$. Let $H_{k,l}(Y)$ denote the $\ko$-module
$H_k(Q)\otimes \Lambda_l$. By K\"{u}nneth's formula, $H_j(Y)\cong
\bigoplus_{k+l=j}H_{k,l}(Y)$.

\begin{prop}\label{propHXstruct}
Over a field, there exists a decomposition $H_j(X)\cong
\bigoplus_{k+l=j}H_{k,l}(X)$ and the $\ko$-module homomorphisms
$f_*\colon H_{k,l}(Y)\to H_{k,l}(X)$ with the following
properties:
\begin{enumerate}
\item If $k>l$, then $f_*\colon H_{k,l}(Y)\to
H_{k,l}(X)$ is an isomorphism. In particular, $H_{k,l}(X)\cong
H_k(Q)\otimes\Lambda_l$.

\item If $k<l$, there exists an isomorphism $H_{k,l}(X)\cong H_k(Q,\dd
Q)\otimes \Lambda_l$.

\item If $k<n$, the module $H_{k,k}(X)$ fits in the exact sequence
\[
0\rightarrow\Ea{X}^{\infty}_{k,k}\rightarrow H_{k,k}(X)\rightarrow
H_k(Q,\dd Q)\otimes\Lambda_k \rightarrow 0.
\]

\item $H_{n,n}(X)\cong\ko$.
\end{enumerate}
There holds bigraded Poincare duality: $H_{k,l}(X)\cong
H_{n-k,n-l}(X)$.
\end{prop}

%
%
%
%
%
%
%

\section{Preliminary computations}\label{SecComput}

%
%
%

\subsection{Orientations}

We use the notation $I\less{k}J$ whenever the simplices $I,J\in S$
satisfy $I<J$ and $|J|-|I|=k$. For each pair $I\less{2}J$, there
are exactly two simplices $J'\neq J''$ between them: $I\less{1}
J',J'' \less{1} J$. For every simplicial poset, there exists a
``sign convention'' which means that we can associate an incidence
number $\inc{J}{I}=\pm 1$ to any pair $I\less{1}J\in S$ in such
way that the relation $\inc{J}{J'}\cdot
\inc{J'}{I}+\inc{J}{J''}\cdot \inc{J''}{I}=0$ holds for any
$I\less{2}J$.

The choice of a sign convention is equivalent to the choice of
orientations of all nonempty simplices. By the orientation of a
simplex $I$ in an abstract simplicial poset we mean the rule which
tells whether a given total ordering of the vertices of $I$ is
positive or negative, so that even permutations of the order
preserve the sign and odd permutations change it.
If $I\less{1}J$, then there is exactly one vertex $i$ of $J$ which
is not in $I$. Given the orientations of simplices $I$ and $J$,
and given some positive ordering $i_1<\ldots<i_s$ of the vertices
of $I$, we set $\inc{J}{I}$ to be $+1$ if $i<i_1<\ldots<i_s$ is a
positive ordering on $\ver(J)$, and $-1$ if it is negative. The
construction works in the opposite direction in an obvious way:
incidence signs determine orientations of all simplices by
induction.

Fix arbitrary orientations of the orbit space $Q$ and the torus
$T^n$. Together they define an orientation of $Y=Q\times T^n$ and
$X=Y/\simc$. Also choose an \emph{omniorientation}, which means
the orientations of all characteristic submanifolds $X_{\{i\}}$. A
choice of an omniorientation determines the characteristic values
$\omega_i\in H_1(T^n;\Zo)$ without ambiguity of sign. To perform
explicit calculations with the spectral sequences $\Ea{X}^*$ and
$\Ea{Y}^*$ we also need to orient all faces of $Q$.

\begin{con}
The orientation of a simplex $I\in S_Q$ determines the orientation
of a face $F_I\subset Q$ by the following convention.

Let $i_1,\ldots,i_{n-q}$ be the vertices of $I$, listed in a
positive order. The face $F_I$ lies in the intersection of facets
$F_{i_1},\ldots,F_{i_{n-q}}$. The normal bundles $\nu_{i}$ to
facets $F_{i}$ have natural orientations, in which inward normal
directions are positive. Orient $F_I$ in such way that
$T_xF_I\oplus\nu_{i_1}\oplus\ldots\oplus\nu_{i_{n-q}}\cong T_xQ$
is positive in the orientation of $Q$.
\end{con}

Thus there are distinguished elements $[F_I]\in H_{\dim
F_I}(F_I,\dd F_I)$. By checking the signs one can prove that the
maps
\[
m^0_{I,J}\colon H_{\dim F_I}(F_I,\dd F_I)\to H_{\dim F_J}(F_J,\dd
F_J)
\]
(see \eqref{eqMattachingMap}) send $[F_I]$ to
$\inc{J}{I}\cdot[F_J]$.

The choice of omniorientation and orientations of simplices
determines the orientation of each orbit $T^n/T_I$ by the
following convention.

\begin{con}
Let $i_1,\ldots,i_{n-q}$ be the vertices of $I$, listed in a
positive order. The module $H_1(T^n/T_I)$ is naturally identified
with $\Lambda_1/L_I$, where $L_I$ is a submodule generated by
$\omega_{i_1},\ldots,\omega_{i_{n-q}}\in \Lambda_1=H_1(T^n)$. The
basis $[\gamma_1],\ldots,[\gamma_q]\in H_1(T^n/T_I)$,
$[\gamma_l]=\gamma_l+L_I$ is said to be positive if the basis
$(\omega_{i_1},\ldots,\omega_{i_{n-q}},\gamma_1,\ldots,\gamma_q)$
is positive in $\Lambda_1$. This orientation of $T^n/T_I$
determines a distinguished fundamental cycle $\Omega_I\in
H_q(T^n/T_I)$.
\end{con}

The omniorientation and the orientation of $S$ together determine
the class of each face submanifold: $[X_I]=[F_I]\otimes\Omega_I$.
Note that both orientations $[F_I]$ and $[\Omega_I]$ depend on the
orientation of $I$ by construction. Thus $[X_I]$ does not actually
depend on the sign convention on $S_Q$ and depends only on the
omniorientation.

%
%
%

\subsection{Arithmetics of torus quotients}
Let us fix some coordinate representation of the torus
$T^n=T^{(\{1\})}\times\ldots\times T^{(\{n\})}$, where each
$T^{(\{j\})}$ is a $1$-dimensional torus with a chosen
orientation. For a subset $A=\{j_1<\ldots<j_q\}\subseteq [n]$ we
denote the coordinate subtorus $T^{(\{j_1\})}\times\ldots\times
T^{(\{j_q\})}\subseteq T^n$ by $T^{(A)}$.

The coordinate splitting gives a positive basis $e_1,\ldots,e_n$
of the module $\Lambda_1=H_1(T^n)$, where $e_j$ is the fundamental
class of $T^{(\{j\})}$. For a vertex $i\in \ver(S)$ let
$(\lambda_{i,1},\ldots,\lambda_{i,n})$ denote the coordinates of
$\omega_i\in \Lambda_1$ in this basis.

\begin{lemma}\label{lemmaToricCoefficient}
Let $I\in S_Q$, $I\neq\minel$ be a simplex with the vertices
$\{i_1,\ldots,i_{n-q}\}$ listed in a positive order. Let
$A=\{j_1<\ldots<j_q\}\subset [n]$ be a subset of indices, and let
$e_A=e_{j_1}\wedge\ldots\wedge e_{j_q}\in H_q(T^n;\Zo)$ be the
fundamental class of $T^{(A)}$. Consider the map $\varrho\colon
T^n\to T^n/T_I$. Then $\varrho_*(e_A) = C_{I,A}\Omega_I\in
H_q(T^n/T_I;\Zo)$. The constant $C_{I,A}$ is equal to
\[
\sgn_A\det\left(\lambda_{i,j}\right)_{\begin{subarray}{l} i\in
\{i_1,\ldots,i_{n-q}\}\\j\in [n]\setminus A
\end{subarray}},
\]
where $\sgn_A=\pm1$ depends only on $A\subset[n]$. When $q=0$, the
constant $C_{I,A}$ is equal to $\pm1$ depending on the positivity
of the basis $\omega_{i_1},\ldots,\omega_{i_n}$ and coincides with
the sign of the fixed point of the action.
\end{lemma}

\begin{proof}
When $q=0$ the statement is simple. Let $q\neq 0$. Choose vectors
$\gamma_1,\ldots,\gamma_q$ so that $(b_l) =
(\omega_{i_1},\ldots,\omega_{i_{n-q}},\gamma_1,\ldots,\gamma_q)$
is a positive basis of the lattice $H_1(T^n,\Zo)$. Thus
$b_l=\mathcal{U}e_l$ with the matrix $\mathcal{U}$ of the form
\[
\mathcal{U}=\begin{pmatrix}
\lambda_{i_1,1}&\cdots& \lambda_{i_{n-q},1} & * &\cdots & *\\
\lambda_{i_1,2}&\cdots& \lambda_{i_{n-q},2} & * &\cdots & *\\
\vdots &\ddots & \vdots & \vdots & \ddots & \vdots\\
\lambda_{i_1,n}&\cdots& \lambda_{i_{n-q},n} & * &\cdots & *
\end{pmatrix}
\]
We have $\det \mathcal{U}=1$ since both bases are positive.
Consider the inverse matrix $\mathcal{V}=\mathcal{U}^{-1}$. Thus
\[
e_A=e_{j_1}\wedge\ldots\wedge
e_{j_q}=\sum_{M=\{\alpha_1<\ldots<\alpha_q\}\subset[n]}
\det\left(\mathcal{V}_{j,\alpha}\right)_{\begin{subarray}{l} j\in
A\\\alpha\in M \end{subarray}}b_{\alpha_1}\wedge\ldots\wedge
b_{\alpha_q}.
\]
After taking the quotient
$\Lambda/\langle\omega_{i_1},\ldots,\omega_{i_{n-q}}\rangle$ all
summands with $M\neq\{n-q+1,\ldots,n\}$ vanish. When
$M=\{n-q+1,\ldots,n\}$, the element $b_{n-q-1}\wedge\ldots\wedge
b_{n} = \gamma_1\wedge\ldots \wedge\gamma_q$ goes to $\Omega_I$.
Thus
\[
C_{I,A} =
\det\left(\mathcal{V}_{j,\alpha}\right)_{\begin{subarray}{l}j\in
A\\\alpha\in\{n-q+1,\ldots,n\}\end{subarray}}.
\]
Now apply Jacobi's identity (see e.g. \cite[Sect.4]{BruShn}):
\[
\det\left(\mathcal{V}_{j,\alpha}\right)_{\begin{subarray}{l} j\in
A\\\alpha\in\{n-q+1,\ldots,n\}\end{subarray}}=
\frac{(-1)^{\sgn}}{\det \mathcal{U}}
\det\left(\mathcal{U}_{r,s}\right)_{\begin{subarray}{l}r\in
\{1,\ldots,n-q\}\\s\in [n]\setminus A\end{subarray}}
\]
where $\sgn=\sum_{r=1}^{n-q}r +\sum_{s\in [n]\setminus A}s$. Since
first $n-q$ columns of $\mathcal{U}$ are exactly the vectors
$\lambda_{i,j}$, this observation completes the proof .
\end{proof}

%
%
%

\subsection{Face ring and linear system of parameters}

Recall the definition of a face ring of a simplicial poset $S$
(see \cite{StPos} or \cite{BPposets}). For $I_1,I_2\in S$ let
$I_1\vee I_2$ denote the set of least upper bounds, and $I_1\cap
I_2\in S$ be the intersection of simplices (the intersection is
well-defined and unique in the case when $I_1\vee
I_2\neq\emptyset$).

\begin{defin}
The face ring $\ko[S]$ is the quotient ring of $\ko[v_I\mid I\in
S]$, $\deg v_I = 2|I|$ by the relations
\[
v_{I_1}\cdot v_{I_2}=v_{I_1\cap I_2}\cdot\sum_{J\in I_1\vee
I_2}v_J,\qquad v_{\emptyset}=1.
\]
The sum over an empty set is assumed to be $0$.
\end{defin}

Let $[m]=\{1,\ldots,m\}$ be the set of vertices of $S$ and let
$\ko[m]=\ko[v_1,\ldots,v_m]$ be the graded polynomial ring with
$\deg v_i = 2$. The ring homomorphism $\ko[m]\to \ko[S]$ sending
$v_i$ to $v_i$ defines a structure of $\ko[m]$-module on $\ko[S]$.

A characteristic function on $Q$ determines the set of linear
forms $\{\theta_1,\ldots,\theta_n\}\subset \ko[S_Q]$, where
$\theta_j=\sum_{i\in\ver(S_Q)} \lambda_{i,j}v_i$. If $J\in S$ is a
maximal simplex, $|J|=n$, then
\begin{equation}\label{eqMinorCharMap}
\mbox{the matrix }(\lambda_{i,j})_{\begin{subarray}{l} i\leqslant
J\\j\in [n]\end{subarray}}\mbox{ is invertible over }\ko
\end{equation}
by the $(\ast)$-condition. This condition is equivalent to the
statement that the sequence $\{\theta_1,\ldots,\theta_n\}$ is a
linear system of parameters in $\ko[S]$ (see
e.g.\cite[Lm.3.5.8]{BPnew}). It generates an ideal
$(\theta_1,\ldots,\theta_n)\subset \ko[S]$ denoted by $\Theta$.

A face ring $\ko[S]$ is an algebra with straightening law (see,
e.g. \cite[\S.3.5]{BPnew}). As a $\ko$-module, it has an additive
basis
\[
\{P_\sigma=v_{I_1}\cdot v_{I_2}\cdot\ldots\cdot v_{I_t}\mid
\sigma=(I_1\leqslant I_2\leqslant\ldots\leqslant I_t\in S)\}.
\]

\begin{lemma}\label{lemmaFaceRingBasis}
The elements $[v_I]=v_I+\Theta$ span the $\ko$-module
$\ko[S]/\Theta$.
\end{lemma}

\begin{proof}
Take any element $P_\sigma$ with $|\sigma|\geqslant 2$. Using
relations in the face ring, we can express
$P_{\sigma}=v_{I_1}\cdot \ldots\cdot v_{I_t}$ as $v_{i}\cdot
v_{I_1\setminus i}\cdot \ldots\cdot v_{I_t}$ for a vertex
$i\leqslant I_1$. Indeed, for every $J\in i\vee (I_1\setminus i)$
except $I_1$, the product $v_J\cdot v_{I_2}$ vanishes.

The element $v_i$ can be expressed as $\sum_{i'\nleq
I_t}a_{i'}v_{i'}$ modulo $\Theta$ according to
\eqref{eqMinorCharMap} (exclude all $v_i$ corresponding to the
vertices of some maximal simplex $J\geqslant I_t$). Thus
$v_iv_{I_t}$ is expressed as a combination of $v_{I_t'}$ with
$I_t'\gess{1}I_t$. Therefore, up to ideal $\Theta$, the element
$P_{\sigma}$ is expressed as a linear combination of elements
$P_{\sigma'}$ which have either smaller length $t$ (in case
$|I_1|=1$) or smaller $I_1$ (in case $|I_1|>1$). By iterating this
descending process, we express the element $P_{\sigma}+\Theta\in
\ko[S]/\Theta$ as a linear combination of $[v_I]$.
\end{proof}


%
%
%
%
%
%
%

\section{Linear relations on face classes}\label{SecFaceSubmfds}

Let $H^*_T(X)$ be a $T^n$-equi\-va\-ri\-ant cohomology ring of
$X$. Any proper face of $Q$ is acyclic, therefore any face has a
vertex. Hence, there is an injective homomorphism
\[
\ko[S_Q]\hookrightarrow H_T^*(X),
\]
which sends $v_I$ to the cohomology class, equivariant Poincare
dual to $[X_I]$ (see \cite[Lm.6.4]{MasPan}). The inclusion of a
fiber in the Borel construction, $X\to X\times_TET^n$, induces the
ring homomorphism $H_T^*(X)\to H^*(X)$. The subspace $H_*(X)$,
Poincare dual to the image of the homomorphism
\begin{equation}\label{eqFaceToCohomology}
g\colon\ko[S_Q]\hookrightarrow H_T^*(X)\to H^*(X)
\end{equation}
is generated by the elements $[X_I]$, thus coincides with the
submodule $\bigoplus_q\Ea{X}^{\infty}_{q,q}\subset H_*(X)$. We
call the classes $[X_I]$ the \emph{face classes (or face cycles)}
of $X$.

Note that the elements $[X_I]=[F_I]\otimes\Omega_I$ can also be
considered as the free generators of the $\ko$-module
\[
\bigoplus_q\E{X}^1_{q,q}=\bigoplus_q\bigoplus_{|I|=n-q}
H_q(F_I,\dd F_I)\otimes H_q(T^n/T_I).
\]
In the following let $\langle[X_I]\rangle$ denote the free
$\ko$-module generated by the elements~$[X_I]$, $I\in S_Q$.

\begin{prop}\label{propTwoTypesOfRels}
Let $C_{I,A}$ be the constants defined in Lemma
\ref{lemmaToricCoefficient}. The submodule of $H_*(X)$ generated
by the face classes $[X_I]$ has relations of the following two
types:
\begin{enumerate}

\item For each $J\in S$, $|J|=n-q-1$, and $A\subset [n]$, $|A|=q$
there is a relation $R_{J,A}=0$ where
\[
R_{J,A}=\sum_{I, I\gess{1}J}\inc{I}{J}C_{I,A}[X_I].
\]

\item Let $q\leqslant n-2$ and let $\beta\in H_q(\dd Q)$ be a
homology class lying in the image of the connecting homomorphism
$\delta_{q+1}\colon H_{q+1}(Q,\dd Q)\to H_q(\dd Q)$. Let
$\sum_{I,|I|=n-q}B_I[F_I]$ be a cellular chain representing
$\beta$ (such representation exists since every face of $\dd Q$ is
acyclic, thus may be considered as a homological cell). Then, for
each $A\subset[n]$, $|A|=q$ we have a relation $R'_{\beta,A}=0$,
where
\[
R'_{\beta,A}=\sum_{I, |I|=n-q}B_IC_{I,A}[X_I].
\]
\end{enumerate}
\end{prop}

\begin{proof}
The proof follows from the structure of homological spectral
sequences of $X$ and $Y$. The module $\bigoplus_q\E{X}^1_{q,q}$ is
freely generated by $[X_I]$. Relations on $[X_I]$ in $H_*(X)$
appear as the images of the differentials hitting
$\bigoplus_q\E{X}^1_{q,q}$. The relation of first type $R_{J,A}$
is the image of the generator
\[
[F_J]\otimes [T^{(A)}]\in H_{q+1}(F_J,\dd F_J)\otimes
H_q(T^n/T_J)\subset \E{X}^1_{q+1,q}
\]
under the differential $\dif{X}^1\colon \E{X}^1_{q+1,q}\to
\E{X}^1_{q,q}$. Thus relations of the first type span the image of
the first differentials hitting $\E{X}^1_{q,q}$.

Let us prove that images of higher differentials are generated by
$R'_{\beta,A}$. Higher differentials $\dif{Q}^{\geqslant 2}$
coincide with $\delta_{*+1}\colon H_{*+1}(Q,\dd Q)\to H_*(\dd Q)$
by Proposition \ref{propEQstruct}. The differentials
$\dif{Y}^{\geqslant 2}$ coincide with
$\delta_{*+1}\otimes\id_{\Lambda}$ by Proposition
\ref{propEYstruct}. Thus the image of $\dif{Y}^{\geqslant 2}$ in
$\E{Y}^2_{q,q}$ is generated by the elements $\beta\otimes
[T^{(A)}]$, which are, in turn, the homology classes of the
elements
\[
\left(\sum_{I, |I|=n-q}B_I[F_I]\right)\otimes[T^{(A)}]\in
\E{Y}^1_{q,q}.
\]
By Proposition \ref{propEXstruct}, the differential $\dif{X}^*$
which hits $\E{X}^*_{q,q}$, coincides with the composition of
$\dif{Y}^*$ and inclusion $f_*^2$. The map $f_*^1\colon
\E{Y}^1_{q,q}\to\E{X}^1_{q,q}$ is the sum of the maps
\[
\id\otimes\varrho\colon H_q(F_I,\dd F_I)\otimes H_q(T^n)\to
H_q(F_I,\dd F_I)\otimes H_q(T^n/T_I)
\]
over all simplices $I$ of rank $n-q$. Thus
\[
f_*^2(\beta\otimes [T^{(A)}])=\left[f_*^1\left(\left(\sum_{I,
|I|=n-q} B_I[F_I]\right)\otimes[T^{(A)}]\right)\right]
=R'_{\beta,A}
\]
by Lemma \ref{lemmaToricCoefficient}.
\end{proof}

\begin{rem}
It follows from the spectral sequence that the element
$R'_{\beta,A}\in \bigoplus_q\E{X}^2_{q,q}$ does not depend on a
cellular chain, representing $\beta$. Proposition
\ref{propEXstruct} also implies that relations $\{R'_{\beta,A}\}$
are linearly independent in $\E{X}^2_{q,q}$ when $\beta$ runs over
some basis of $\im\delta_{q+1}$ and $A$ runs over all subsets of
$[n]$ of cardinality $q$.
\end{rem}

Next we want to check that relations of the first type are exactly
the relations in the quotient ring $\ko[S_Q]/\Theta$.

\begin{prop}\label{propIsomQuotients}
Let $\varphi\colon\langle[X_I]\rangle\to \ko[S_Q]$ be the degree
reversing linear map, which sends the generator $[X_I]$ to $v_I$.
Then $\varphi$ descends to the isomorphism of $\ko$-modules
\[
\tilde{\varphi}\colon \langle[X_I]\rangle/\langle
R_{J,A}\rangle\to \ko[S_Q]/\Theta.
\]
\end{prop}

\begin{proof}
(1) First we prove that $\tilde{\varphi}$ is well defined by
showing that the element
\[
\varphi(R_{J,A})=\sum\nolimits_{I, I\gess{1}J}
\inc{I}{J}C_{I,A}v_I\in\ko[S_Q]
\]
lies in $\Theta$. Let $s=|J|$ and, consequently, $|I|=s+1$,
$|A|=n-s-1$. Let $[n]\setminus A=\{\alpha_1<\ldots<\alpha_{s+1}\}$
and let $\{j_1,\ldots,j_s\}$ be the vertices of $J$ listed in a
positive order. Consider the $s\times(s+1)$ matrix:
\[
\mathcal{D}=\begin{pmatrix}
\lambda_{j_1,\alpha_1}&\ldots& \lambda_{j_1,\alpha_{s+1}}\\
\vdots &\ddots & \vdots\\
\lambda_{j_s,\alpha_1}&\ldots& \lambda_{j_s,\alpha_{s+1}}
\end{pmatrix}
\]
Denote by $\mathcal{D}_l$ the square submatrix obtained from
$\mathcal{D}$ by deleting $l$-th column and let
$a_l=(-1)^{l+1}\det \mathcal{D}_l$. We claim that
\[
\varphi(R_{J,A}) = \pm
v_J\cdot(a_1\theta_{\alpha_1}+\ldots+a_{s+1}\theta_{\alpha_{s+1}}).
\]
Indeed, after expanding each $\theta_l$ as $\sum_{i\in
\ver(S)}\lambda_{i,l}v_i$, all elements of the form $v_Jv_{i}$
with $i<J$ cancel (the coefficients at these terms are
determinants of matrices with two coinciding rows). Other terms
give
\[
\sum_{I,I\gess{1}J, i=I\setminus
J}(a_1\lambda_{i,\alpha_1}+\ldots+a_{s+1}\lambda_{i,\alpha_{s+1}})v_I.
\]
The coefficient at $v_I$ is equal to the determinant of the matrix
\begin{equation}\label{eqMatrixBig}
\begin{pmatrix}
\lambda_{i,\alpha_1}&\ldots& \lambda_{i,\alpha_{s+1}}\\
\lambda_{j_1,\alpha_1}&\ldots& \lambda_{j_1,\alpha_{s+1}}\\
\vdots &\ddots & \vdots\\
\lambda_{j_s,\alpha_1}&\ldots& \lambda_{j_s,\alpha_{s+1}}
\end{pmatrix}
\end{equation}
by the cofactor expansion along the first row. This determinant is
equal to $\sgn_A \inc{I}{J} C_{I,A}$ by the definition of
$C_{I,A}$. Indeed, the number $C_{I,A}$ was defined as the
determinant for some positive ordering of vertices of $I$. The
ordering $i<j_1<\ldots<j_s$ (used to order the rows of matrix
\eqref{eqMatrixBig}) is either positive or negative depending on
the incidence sign $\inc{I}{J}$.

(2) $\tilde{\varphi}$ is surjective by Lemma
\ref{lemmaFaceRingBasis}.

(3) Ranks of both spaces are equal. Indeed, $\dim
\langle[X_I]\mid|I|=n-q\rangle/\langle R_{J,A}\rangle=\dim
\E{X}^2_{q,q}=h'_{n-q}(S_Q)$ by Proposition \ref{propEXstruct}. By
Proposition \ref{propSisBuch} the poset $S_Q$ is Buchsbaum. Thus
$\dim(\ko[S_Q]/\Theta)_{n-q}=h'_{n-q}(S_Q)$ by Schenzel's theorem
(see \cite{Sch}, \cite[Ch.II,\S8.2]{St}, or \cite[Prop.6.3]{NS}
for simplicial posets).

(4) If $\ko$ is a field, then we are done. Since the statement
holds over any field, the case $\ko=\Zo$ automatically follows.
\end{proof}

The Poincare duality in $X$ yields

\begin{cor}\label{corKernel}
The map $g\colon \ko[S_Q]\to H^*(X)$ factors through
$\ko[S_Q]/\Theta$ and the kernel of the homomorphism
$\tilde{g}\colon \ko[S_Q]/\Theta\to H^*(X)$ is additively
generated by the elements
\[
L'_{\beta,A}=\sum_{I, |I|=n-q}B_IC_{I,A}v_I,
\]
where $q\leqslant n-2$, $\beta\in\im(\delta_{q+1}\colon
H_{q+1}(Q,\dd Q)\to H_q(\dd Q))$, $\sum_{I, |I|=n-q}B_I[F_I]$ is a
cellular chain in $\dd Q$ representing $\beta$, and $A\subset
[n]$, $|A|=q$.
\end{cor}

\begin{rem}
The ideal $\Theta\subset \ko[S_Q]$ coincides with the image of the
homomorphism $H^{>0}(BT^n)\to H^*_T(X)$. So the fact that $\Theta$
vanishes in $H^*(X)$ is not surprising. An interesting thing is
that $\Theta$ vanishes already in a second term of the spectral
sequence, while other relations in $H^*(X)$ demonstrate the
effects of higher differentials.
\end{rem}

Note, that the elements $R'_{\beta,A}=\sum_{I,
|I|=n-q}B_IC_{I,A}[X_I]\in \E{X}^2_{q,q}$ and
$L'_{\beta,A}=\sum_{I, |I|=n-q}B_IC_{I,A}v_I\in \ko[S_Q]/\Theta$
can be defined for any homology class $\beta\in H_q(\dd Q)$ (not
only for the image of the connecting homomorphism $\delta_{q+1}$).

Recall that the \emph{socle} of a module $\mathcal{M}$ over the
polynomial ring $\ko[m]$ is the submodule
\[
\soc\mathcal{M} \eqd \{y \in \mathcal{M} \mid \ko[m]^{+}\cdot y =
0\},
\]
where $\ko[m]^+$ is the maximal graded ideal of $\ko[m]$.

\begin{thm}\label{thmSocle}
For every $\beta\in H_q(\dd Q)$, $q\leqslant n-1$ and $A\subset
[n]$, $|A|=q$ the element $L'_{\beta,A}\in \ko[S_Q]/\Theta$ lies
in a socle of $\ko[S_Q]/\Theta$.
\end{thm}

We postpone the proof to Section \ref{SecCollar}.

%
%
%
%
%
%
%

\section{Non-face cycles of $X$}\label{SecOtherCycles}

\subsection{Spine and diaphragm cycles}
In this section we give a geometrical description of homology
classes of $Q$ different from face classes.

\begin{con}\label{conSpineCycles}
Let $\eta\in H_k(Q)$ be a cycle of $Q$ and let $a\in H_l(T^n)$,
$l<k$ be a cycle of $T^n$ represented by a subtorus
$T^{(a)}\subset T^n$. They determine the homology class
$\eta\otimes[T^{(a)}]\in H_{k,l}(Y)\cong H_k(Q)\otimes H_l(T^n)$.
Thus they determine the class $\spi_{\eta,a}\in H_{k+l}(X)$ via
the isomorphism $f_*\colon H_{k,l}(Y)\to H_{k,l}(X)$ asserted by
pt.(1) of Proposition \ref{propHXstruct}. The cycles of this form
(which are the elements of $H_{k,l}(X)$ for $k>l$) will be called
\emph{spine cycles (or spine classes)}.

Suppose that $\eta$ is represented by an embedded pseudomanifold
$N\subset Q$. We may assume that $N$ lies in the interior of $Q$.
Then the spine cycle $\spi_{\eta,a}$ is represented by an embedded
pseudomanifold $N\times T^{(a)}\subset Q^{\circ}\times T^n\subset
X$.
\end{con}

\begin{con}
Suppose $k<n$ and let $\zeta\in H_k(Q,\dd Q)$ be a relative
homology class. Assume for simplicity that $\zeta$ is represented
by a submanifold (or, generally, embedded pseudomanifold)
$L\subset Q$ of dimension $k$ with boundary $\dd L\subset \dd Q$
(which may be empty). Every proper face of $Q$ is acyclic, thus
can be considered as a homological cell of $\dd Q$. Therefore
without lost of generality we may assume $\dd L\subset Q_{k-1}$.
We also assume that $L\setminus \dd L\subset Q\setminus \dd Q$.

For each class $a\in H_l(T^n)$, represented by an $l$-dimensional
subtorus $T^{(a)}$, consider the subset $Z_{L,a}=(L\times
T^{(a)})/\simc\subset X$.
\end{con}

\begin{prop}\label{propDiaphragmDef}\mbox{}

\begin{itemize}
\item If $l\geqslant k$, the subset $Z_{L,a}$ is a pseudomanifold.
Thus it represents a well-defined element $\dia_{L,a}\in
H_{k+l}(X)$ which will be called a diaphragm class (or diaphragm
cycle).

\item If $l>k$, the class $\dia_{L,a}\in H_{k+l}(X)$ depends only on
the class $\zeta=[L]\in H_k(Q,\dd Q)$ but not on the particular
representative $L$.
\end{itemize}
\end{prop}

\begin{proof}
The set $((L\setminus \dd L)\times T^{(a)})/\simc = (L\setminus
\dd L)\times T^{(a)}$ is a manifold of dimension $k+l$. The
exceptional locus $(\dd L\times T^{(a)})/\simc$ has dimension at
most $k+l-2$. Indeed, we have $\dd L\subset Q_{k-1}$, thus, under
the projection map $(\dd L\times T^{(a)})/\simc\to\dd L$, every
point $x\in \dd L$ has a preimage of the form $T^{(a)}/T_I$ with
$|I|\geqslant n-k+1$. This set has dimension at most $l-1$ since
$l+(n-k+1)>n$. Thus the total dimension of exceptional locus is at
most $\dim \dd L+(l-1)=k+l-2$.

The second statement can be proved similarly. Let $(L_1,\dd L_1)$
and $(L_2,\dd L_2)$ be two manifolds representing the same element
$\zeta\in H_k(Q,\dd Q)$. There exists a pseudomanifold bordism
between them, i.e. a pseudomanifold $\Xi$ of dimension $k+1$ with
boundary and a map $\phi\colon \Xi\to Q$ such that $L_1$, $L_2$
are disjoint submanifolds of $\partial \Xi$, the restriction of
$\phi$ to $L_\epsilon$ is the inclusion of
$L_\epsilon\hookrightarrow Q$, and $\phi(\partial \Xi\setminus
(L_1^{\circ}\sqcup L_2^{\circ}))\subset \partial Q$ (this follows
from the geometrical definition of homology, see \cite[App.
A.2]{RuSa}). Again, we may assume that $\phi(\partial \Xi\setminus
(L_1^{\circ}\sqcup L_2^{\circ}))\subset Q_k$. Similar to the first
statement, we can consider the space $(\Xi\times T^{(a)})/\simc$
of dimension $k+l+1$. This space is a pseudomanifold with
boundary, and the boundary is exactly the difference
$Z_{L_1,a}-Z_{L_2,a}$. Thus $\dia_{L_1,a}=\dia_{L_2,a}$ in
$H_*(X)$.
\end{proof}

Thus for $k<l$ there is a well defined homology class
$\dia_{\zeta,a}\eqd\dia_{L,a}\in H_{k,l}(X)$ depending on
$\zeta\in H_k(Q,\dd Q)$ and $a\in H_l(T^n)$. These classes span
the homology modules $H_{k,l}(X)$ for $k<l$ and correspond to
pt.(2) of Proposition \ref{propHXstruct}.

When $k=l<n$ we call the classes $\dia_{L,a}$ \emph{extremal
diaphragm classes}. In this case the situation is a bit different:
the classes $\dia_{L,a}$ depend not only on the homology class of
$L$ but on the representative $L$ itself. Nevertheless, if $L_1$
and $L_2$ represent the same class in $H_k(Q,\dd Q)$, then the
classes $\dia_{L_1,a}, \dia_{L_2,a}\in H_{k,k}(X)$ coincide modulo
face classes, as proved below. Our goal is to derive exact
formulas, thus we restrict to the case, when $a\in H_l(T^n)$ is
represented by a coordinate subtorus $T^{(A)}$ for $A\subset [n]$,
$|A|=l$.

\begin{con}
Let $\phi_\epsilon\colon (L_\epsilon,\partial L_\epsilon)\to
(Q,\partial Q)$, $\epsilon=1,2$, be two pseudomanifolds
representing the same element $\zeta\in H_k(Q,\partial Q)$, $k<n$.
As in the proof of Proposition \ref{propDiaphragmDef} consider a
pseudomanifold bordism $(\Xi,\partial \Xi)$ between $L_1$ and
$L_2$, and the map $\phi\colon \Xi\to Q$, which sends the boundary
$\dd L$ into the union of $L_1$, $L_2$ and $Q_k$. The skeletal
stratification of $Q$ induces a stratification on $\Xi$. The
restriction of the map $\phi$ sends $\Xi_{k-1}$ to $Q_{k-1}$. Let
$\delta$ be the connecting homomorphism
\[
\delta\colon H_{k+1}(\Xi,\dd \Xi)\to H_k(\dd \Xi,\Xi_{k-1})
\]
in the long exact sequence of the triple $(\Xi,\dd
\Xi,\Xi_{k-1})$. Consider the sequence of homomorphisms

\begin{multline*}
H_{k+1}(\Xi,\dd \Xi)\xrightarrow{\delta} H_{n-1}(\dd
\Xi,\Xi_{k-1}) \cong \\ H_k(L_1,\dd L_1)\oplus H_k(L_2,\dd
L_2)\oplus H_k(\dd \Xi\setminus
(L_1^{\circ}\cup L_2^{\circ}), \dd \Xi_{k-1}) \xrightarrow{\id\oplus\id\oplus \phi_*}\\
H_k(L_1,\dd L_1)\oplus H_k(L_2,\dd L_2)\oplus H_k(Q_k, Q_{k-1}).
\end{multline*}
It sends the fundamental cycle $[\Xi]\in H_{k+1}(\Xi,\dd \Xi)$ to
the element
\begin{equation}\label{eqRelPseudBordism}
\left([L_1],-[L_2],\sum_{I, \dim F_I=k}D_I[F_I]\right)
\end{equation}
of the group $H_k(L_1,\dd L_1)\oplus H_k(L_2,\dd L_2)\oplus
H_k(Q_k, Q_{k-1})$, for some numbers $D_I\in\ko$.
\end{con}

\begin{prop}\label{propRelOnDiaphragm}
Let $L_1,L_2$ be two manifolds representing the same class
$\zeta\in H_k(Q,\partial Q)$, $k<n$. Consider any subset $A\subset
[n]$, $|A|=k$ and let $a\in H_k(T^n)$ be the fundamental class of
the coordinate subtorus $T^{(A)}$. Then there is a relation in
$H_{2k}(X)$:
\begin{equation}\label{eqZeroCombination}
\dia_{L_1,a}-\dia_{L_2,a} + \sum_{I, \dim F_I=k}D_IC_{I,A}[X_I]=0.
\end{equation}
The numbers $D_I$ are given by \eqref{eqRelPseudBordism}, and the
numbers $C_{I,A}$ were defined in
Lemma~\ref{lemmaToricCoefficient}.
\end{prop}

\begin{proof}
Consider the space $(\Xi\times T^{(A)})/\simc'$ and the map
$\phi\times\iota\colon (\Xi\times T^{(A)})/\simc'\to X=(Q\times
T^n)/\simc$, where the relation $\sim'$ is induced from $\sim$ by
the map $\phi$, and $\iota\colon T^{(A)}\to T^n$ is the inclusion
map. The space $(\Xi\times T^{(A)})/\simc'$ is a pseudomanifold
with boundary, and its boundary represents the element
$\dia_{L_1,A}-\dia_{L_2,A} + \sum_{I, \dim F_I=k}D_IC_{I,A}[X_I]$
in $H_{2k}(X)$ by Lemma \ref{lemmaToricCoefficient}. Thus this
element vanishes in homology.
\end{proof}

Therefore, up to face classes, the middle homology group
$H_{k,k}(X)$ coincides with $H_k(Q,\dd Q)\otimes H_k(T^n)$ for
$k<n$. This was stated in equivalent form in pt.(3) of Proposition
\ref{propHXstruct}.

%
%
%

\subsection{Integral coefficients}
Proposition \ref{propHXstruct} was stated only over a field. On
the other hand, the geometrical constructions of the previous
subsection determine the additive homomorphisms
\[
\bigoplus_{k>l}H_k(Q)\otimes H_l(T^n)\to H_*(X),\qquad
\bigoplus_{k\leqslant l} H_k(Q,\dd Q)\otimes H_l(T^n)\to H_*(X)
\]
over $\Zo$ as well. Combining this with the inclusion of face
cycles
\[
\langle[X_I]\rangle/(R_{J,A}, R'_{\beta,A})\hookrightarrow H_*(X)
\]
we get a map
\begin{equation}\label{eqLargeHomomorph}
\left(\bigoplus_{k>l}H_k(Q)\otimes
H_l(T^l)\right)\oplus\left(\bigoplus_{k\leqslant l} H_k(Q,\dd
Q)\right)\oplus\left(\langle[X_I]\rangle/(R_{J,A},
R'_{\beta,A})\right)\to H_*(X)
\end{equation}
which is well defined for any coefficients and is an isomorphism
over a field. Now it follows by induction from the universal
coefficient theorem, that the map \eqref{eqLargeHomomorph} is an
isomorphism over $\Zo$. This proves

\begin{prop}\label{propHXstructIntegers}
Proposition \ref{propHXstruct} holds over $\Zo$. Homology groups
of $X$ are generated by face classes, spine classes, and diaphragm
classes. The groups $H_{k,l}(X)$ for $k>l$ are generated by spine
classes; the groups $H_{k,l}(X)$ for $k<l$ are generated by
non-extremal diaphragm classes; the short exact sequence
\[
0\rightarrow\Ea{X}^{\infty}_{k,k}\rightarrow H_{k,k}(X)\rightarrow
H_k(Q,\dd Q)\otimes\Lambda_k \rightarrow 0.
\]
identifies the quotient of $H_{k,k}(X)$ by the face classes with
the group of extremal diaphragm classes. This short exact sequence
splits, but the splitting is not canonical.
\end{prop}

%
%
%

\subsection{Auxiliary cycles and relations}\label{subsecConventSpineToDiaphragm}
In construction \ref{conSpineCycles} we defined the cycles
$\spi_{\eta,a}$ for each $\eta\in H_k(Q)$ and $a\in H_l(T^n)$
under assumption that $k>l$. But the same construction can be
applied for any $k$ and $l$. If $\eta$ is represented by a
pseudomanifold $N$ in the interior of $Q$ and $a$ is represented
by a subtorus, then the product $N\times T^{(a)}$ is a submanifold
in $Q^{\circ}\times T^n\subset X$, thus represents an element
$[N\times T^{(a)}]\in H_{k+l}(X)$.

Although for $k\leqslant l$ the element $N\times T^{(a)}$ is not a
spine cycle, we keep denoting it $\spi_{\eta,a}$. For $k<l$ we
have $\spi_{\eta,a}=\dia_{\eta',a}$, where $\eta'$ is the image of
$\eta$ in $H_k(Q,\dd Q)$.

%
%
%
%
%
%
%

\section{Intersections in $H_*(X)$}\label{SecIntersect}

Let $\cap\colon H_k(M)\otimes H_l(M)\to H_{k+l-\dim M}(M)$ denote
the intersection product on a closed manifold $M$, i.e. an
operation Poincare dual to the cup-product in cohomology. From the
geometrical structure of $X$ (and also from Proposition
\ref{propIsomQuotients}) follows

\begin{prop}\label{propIntersectFace}
If $I_1, I_2\in S_Q$, then $[X_{I_1}]\cap [X_{I_2}]=[X_{I_1\cap
I_2}] \cap \sum_{J\in I_1\vee I_2}[X_J]$.
\end{prop}

Intersections of spine and diaphragm classes can also be described
geometrically. There exists an intersection product, $\cap\colon
H_{k_1}(Q)\otimes H_{k_2}(Q,\dd Q)\to H_{k_1+k_2-n}(Q)$ dual to
the cohomology product $H^{n-k_1}(Q,\dd Q)\otimes H^{n-k_2}(Q)\to
H^{2n-k_1-k_2}(Q,\dd Q)$. If the classes $\eta\in H_{k_1}(Q)$ and
$\zeta\in H_{k_2}(Q,\dd Q)$ are represented by a smooth
submanifold $N$ and a smooth submanifold with boundary $(L,\dd L)$
respectively, and if $N$ intersects $L$ transversely, then
$\eta\cap\zeta$ is represented by a submanifold $N\cap L$. There
is also an intersection product in homology of torus $\cap\colon
H_{l_1}(T^n)\otimes H_{l_2}(T^n)\to H_{l_1+l_2-n}(T^n)$. From the
geometrical construction of cycles in $X$ we conclude that the
intersection product on $H_*(X)$ has the following structure.

\begin{prop}\label{propIntersections}\mbox{}

\begin{enumerate}
\item The cycles $\spi_{\eta,a}\in H_{k_1,l_1}(X)$, $k_1>l_1$ and
$\dia_{L,b}\in H_{k_2,l_2}(X)$, $k_2\leqslant l_2$ satisfy
\[
\spi_{\eta,a}\cap \dia_{L,b} = \spi_{\eta\cap[L], a\cap b}.
\]
Since $\dim (\eta\cap [L]) = k_1+k_2-n$ and $\dim (a\cap b) =
l_1+l_2-n$, the element $\spi_{\eta\cap[L], a\cap b}$ is either a
spine class (if $k_1+k_2>l_1+l_2$), or a diaphragm class
determined in subsection \ref{subsecConventSpineToDiaphragm} (if
$k_1+k_2\leqslant l_1+l_2$).

\item The cycles $\spi_{\eta',a}\in H_{k_1,l_1}(X)$, $k_1>l_1$ and
$\spi_{\eta'',b}\in H_{k_2,l_2}(X)$, $k_2>l_2$ satisfy
\[
\spi_{\eta',a}\cap \spi_{\eta'',b} = \spi_{\eta'\cap\eta'', a\cap
b}.
\]
The result is a spine cycle.

\item Spine cycles do not meet face cycles: $\spi_{\eta,a}\cap
[X_I] = 0$.
\end{enumerate}
\end{prop}

The proof follows directly from the constructions.

\begin{prop}\label{propFaceIsIdeal}
The linear span of proper face classes $[X_I]$ is an ideal of
$H_*(X)$ with respect to the intersection product.
\end{prop}

\begin{proof}
Suppose $I\neq\minel$ and let $\imath\colon X_I\hookrightarrow X$
be the inclusion of a face submanifold. Let $\kappa$ be a
cohomology class Poincare dual to some diaphragm class $\dia$ in
$X$. Then $[X_I]\cap\dia = [X_I]\capp \kappa =
\imath_*(\imath^*(\kappa)\capp [X_I])$. The class
$\imath^*(\kappa)\capp [X_I]\in H_*(X_I)$ is a linear combination
of face classes since there are no other classes in $H_*(X_I)$.
Thus $[X_I]\cap\dia$ is a linear combination of face classes in
$H_*(X)$.
\end{proof}

\begin{rem}\label{remIntersGeneral}
The description of intersections of diaphragm cycles with
themselves and with face cycles is difficult in general.
Nevertheless, in practice one can use the following trick (cf. the
discussion of a similar problem in \cite[Sect.8]{AMPZ}). Suppose
the task is to compute $\dia_{L,a}\cap [X_I]$. If $L\cap
F_I=\emptyset$, then the intersection product is $0$. If not, find
another submanifold with boundary $L'$ such that $[L']=[L]\in
H_*(Q,\dd Q)$ and $L'\cap F_I=\emptyset$. Then, by Proposition
\ref{propRelOnDiaphragm} we have $\dia_{L,a}=\dia_{L',a}+\Sigma$,
where $\Sigma$ is a linear combination of face classes. Thus we
have
\[
\dia_{L,a}\cap [X_I]=\dia_{L',a}\cap[X_I]+\Sigma\cap[X_I] =
\Sigma\cap[X_I].
\]
The last intersection can be computed by Proposition
\ref{propIntersectFace}.
\end{rem}

%
%
%
%
%
%
%

\section{Examples}\label{SecExamples}

%
%
%

\subsection{A very concrete example}
This example is similar to the one studied in
\cite[Th.3.1]{PodSar}. For $Q$ take a square with triangular hole.
Orientations of facets and values of characteristic function are
assigned to $Q$ as shown on Fig.\ref{figSquareWithHole}, left.

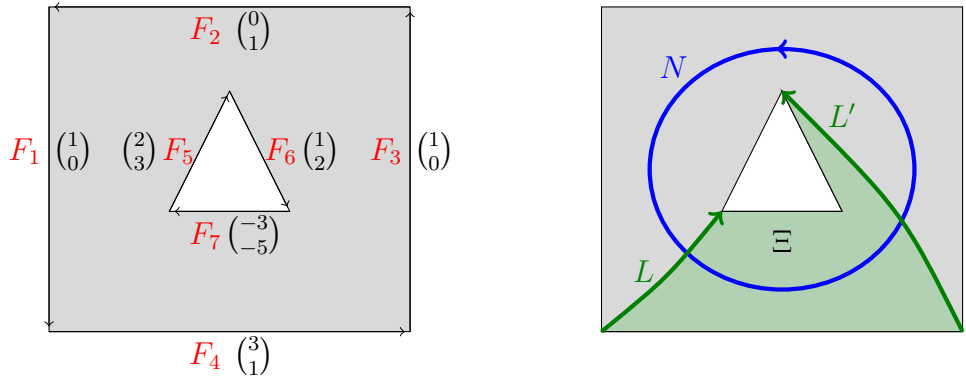
\begin{figure}[h]
\begin{center}
    \begin{tikzpicture}[scale=.8]
        \filldraw[fill=black!15]
        (0,0)--(6,0)--(6,5.4)--(0,5.4)--cycle;
        \filldraw[fill=white]
        (2,2)--(4,2)--(3,4)--cycle;
        \draw[-to,shorten >=2](0,0)--(6,0);
        \draw[-to,shorten >=2](6,0)--(6,5.4);
        \draw[-to,shorten >=2](6,5.4)--(0,5.4);
        \draw[-to,shorten >=2](0,5.4)--(0,0);
        \draw[-to,shorten >=2](4,2)--(2,2);
        \draw[-to,shorten >=2](3,4)--(4,2);
        \draw[-to,shorten >=2](2,2)--(3,4);

        \draw[red] (2.6,-0.4) node{$F_4$};
        \draw (3.4,-0.4) node{$\binom{3}{1}$};

        \draw[red] (-0.4,3) node{$F_1$};
        \draw (0.4,3) node{$\binom{1}{0}$};

        \draw[red] (5.6,3) node{$F_3$};
        \draw (6.4,3) node{$\binom{1}{0}$};

        \draw[red] (2.6,5) node{$F_2$};
        \draw (3.4,5) node{$\binom{0}{1}$};

        \draw[red] (2.6,1.6) node{$F_7$};
        \draw (3.4,1.6) node{$\binom{-3}{-5}$};

        \draw[red] (2.15,3) node{$F_5$};
        \draw (1.5,3) node{$\binom{2}{3}$};

        \draw[red] (3.85,3) node{$F_6$};
        \draw (4.5,3) node{$\binom{1}{2}$};
    \end{tikzpicture}
    \qquad\qquad
    \begin{tikzpicture}[scale=.8]
        \filldraw[fill=black!15]
        (0,0)--(6,0)--(6,5.4)--(0,5.4)--cycle;
        \filldraw[fill=white]
        (2,2)--(4,2)--(3,4)--cycle;
        \filldraw[fill=green!40!black!30]
        (6,0)..controls (5,2)..(3,4)--(4,2)--(2,2)..controls
        (1.1,0.9)..(0,0)--cycle;
        \draw[white] (3.4,-0.4) node{$\binom{3}{1}$};

        \draw[ultra thick, blue] (3,2.7) ellipse (2.2 and 2);
        \draw[<-,ultra thick, blue] (2.9,4.7)--(3.1,4.7);

        \draw[->,ultra thick, green!50!black]  (0,0)..controls (1.1,0.9)..(2,2);
        \draw[->,ultra thick, green!50!black]  (6,0)..controls (5,2)..(3,4);

        \draw[blue] (1.2,4.4) node{$N$};
        \draw[green!50!black] (0.7,1) node{$L$};
        \draw[green!50!black] (4,3.6) node{$L'$};
        \draw[black] (3,1.5) node{$\Xi$};
    \end{tikzpicture}
\end{center}
\caption{Structure of $Q$ and values of characteristic function}
\label{figSquareWithHole}
\end{figure}

Homology groups of the corresponding $4$-dimensional manifold
$X=(Q\times T^2)/\simc$ are as follows.

\textbf{(1) Face classes.} These are the following: the
fundamental class $[X]\in H_4(X)$; the classes of characteristic
submanifolds
\[
[X_1],[X_2],\ldots,[X_7]\in H_{1,1}(X)\subset H_2(X),
\]
which correspond to the sides of $Q$; and the classes of fixed
points
\[
[X_{12}],[X_{23}],[X_{34}],[X_{14}],[X_{56}],[X_{67}], [X_{57}]
\in H_{0,0}(X) = H_0(X),
\]
which correspond to vertices of $Q$. Relations on these classes
are given by Proposition~\ref{propTwoTypesOfRels}. The first-type
relations in $H_2(X)$ are:
\[
[X_1]+[X_3]+3[X_4]+2[X_5]+[X_6]-3[X_7]=0
\]
\[
[X_2]+[X_4]+3[X_5]+2[X_6]-5[X_7]=0
\]
(the coefficients are respectively first and second coordinates of
the values of characteristic function). The first type relations
on the classes $[X_{ij}]$ are encoded by the sides of of the orbit
space. These relations are the following:
\[
[X_{12}]=-[X_{23}]=[X_{34}]=-[X_{14}];
\]
\[
[X_{56}]=[X_{67}]=[X_{57}].
\]
(the signs are due to the signs of fixed points). To find
relations of the second type, we need to pick a homology class in
the image of $\delta_1\colon H_1(Q,\dd Q)\to H_0(\dd Q)$. Take for
example the class, represented by the chain $[F_{14}]-[F_{57}]$.
It gives a relation of the second type
\[
[X_{14}]-[X_{57}]=0.
\]
Thus all fixed points up to sign represent the same generator
$[\pt]\in H_0(X)$. Of course, this follows easily from the
connectivity of $X$, but here we wanted to emphasize the different
nature of two types of relations.

\textbf{(2) Spine cycles.} Consider a submanifold $N\subset Q$
representing the generator $\eta\in H_1(Q)$
(Fig.\ref{figSquareWithHole}, right). Together with the class of a
point $[T^{(\emptyset)}]\in H_0(T^2)$ it determines a spine class
$\spi_{\eta,\emptyset}\in H_{1,0}(X)=H_1(X)$. Geometrically,
$\spi_{\eta,\emptyset}$ is represented by a submanifold $N\subset
Q$ lifted by a zero-section map $Q\hookrightarrow X$.

\textbf{(3) Diaphragm cycles.} Consider the submanifold $L$
representing the generator of $H_1(Q,\dd Q)$ with the boundary
lying in the $0$-skeleton of $Q$ (see Fig.\ref{figSquareWithHole},
right). For each subset $A=\{1\},\{2\},\{1,2\}$ we have a homology
cycle in $H_{1,|A|}(X)$ represented by a pseudomanifold $(L\times
T^{(A)})/\simc$. Thus we have the generators
\[
\dia_{L,1}=[(L\times T^{(\{1\})})/\simc],\qquad
\dia_{L,2}=[(L\times T^{(\{2\})})/\simc]
\]
of $H_{1,1}(X)\subset H_2(X)$ and the generator
\[
\dia_{L,\{12\}}=[(L\times T^2)/\simc]
\]
of $H_{1,2}(X)= H_3(X)$. Let $L'$ be another submanifold
representing the same homology class in $H_1(Q,\dd Q)$ (see Figure
\ref{figSquareWithHole}). Consider a bordism $\Xi$ between $L$ and
$L'$ shown on the figure. We have
\[
\dia_{L,\{12\}} = \dia_{L',\{12\}}
\]
in $H_3(X)$ since $(\Xi\times T^2)/\simc$ is a pseudomanifold
bordism between $(L\times T^2)/\simc$ and $(L'\times T^2)/\simc$.

We have a relation $\delta \Xi = -[L_1]+[L_2]+[F_4]+[F_6]+[F_7]$.
It generates the relations
\[
-\dia_{L,1}+\dia_{L',1}+1[X_4]+2[X_6]-5[X_7]=0
\]
\[
-\dia_{L,2}+\dia_{L',2}+3[X_4]+1[X_6]-3[X_7]=0
\]
in $H_2(X)$. These relations are the boundaries of $(\Xi\times
T^{(\{1\})})/\simc$ and $(\Xi\times T^{(\{2\})})/\simc$
respectively. The coefficients are the complimentary coordinates
of characteristic function: for the cycle encoded by the first
coordinate subtorus, we take the second coordinates of
characteristic function, and vice versa. In this computation we
used the formula for the coefficients $C_{I,A}$ asserted by Lemma
\ref{lemmaToricCoefficient}.

\textbf{(4) Intersections of cycles} can be seen from the picture.
In particular, the cycles $\spi_{N,\emptyset}$ and
$\dia_{L,\{1,2\}}$ are transversal, and their intersection induces
a nondegenerate pairing between $H_1(X)$ and $H_3(X)$.

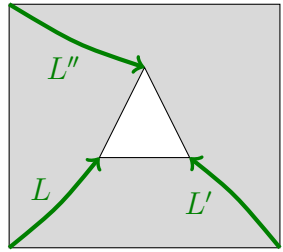
\begin{figure}[h]
\begin{center}
    \begin{tikzpicture}[scale=.6]
        \filldraw[fill=black!15]
        (0,0)--(6,0)--(6,5.4)--(0,5.4)--cycle;
        \filldraw[fill=white]
        (2,2)--(4,2)--(3,4)--cycle;

        \draw[->,ultra thick, green!50!black]  (0,0)..controls (1.1,0.9)..(2,2);
        \draw[->,ultra thick, green!50!black]  (6,0)..controls (5,1.2)..(4,2);
        \draw[->,ultra thick, green!50!black]  (0,5.4)..controls (1.5,4.5)..(3,4);

        \draw[green!50!black] (0.7,1.3) node{$L$};
        \draw[green!50!black] (4.2,1) node{$L'$};
        \draw[green!50!black] (1.2,4) node{$L''$};
    \end{tikzpicture}
\end{center}
\caption{Different representatives of diaphragm classes}
\label{figSquareWithHole2}
\end{figure}

For a nontrivial example, let us compute the intersection of
$\dia_{L,1}$ with $\dia_{L,2}$ to demonstrate the idea sketched in
Remark \ref{remIntersGeneral}. Consider the auxiliary intervals
$L'$ and $L''$ shown on Fig.\ref{figSquareWithHole2}. Similar to
the calculations above we have
\[
\dia_{L,1} = \dia_{L',1}+1[X_4]-5[X_7]
\]
\[
\dia_{L,2} = \dia_{L'',2}-1[X_1]-2[X_5]
\]
Thus $\dia_{L,1}\cap\dia_{L,2} = (\dia_{L',1}+[X_4]-5[X_7])\cap
(\dia_{L'',2}-[X_1]-2[X_5]) =
-[X_4]\cap[X_1]+10[X_7]\cap[X_5]=-[X_{14}]+10[X_{57}]=9[pt]\in
H_0(X)$.

%
%
%

\subsection{Toric origami manifolds}
In this subsection we apply the general method to a class of toric
origami manifolds and derive some results proved in \cite{AMPZ} in
a different way.

Toric origami manifolds (see \cite{SGP},\cite{MasPark}) appeared
in differential geometry as generalizations of symplectic toric
manifolds. The precise geometrical definition is in most part
irrelevant to our study. Essential are the following properties:
orientable toric origami manifold $X$ is a manifold with locally
standard torus action; its orbit space $Q=X/T^n$ is homotopy
equivalent to a graph $\Gamma$, and inclusion of any face in $Q$
is homotopy equivalent to the inclusion of a subgraph in $\Gamma$.

As usual, there is a principal torus bundle $Y\to Q$ such that
$X=Y/\simc$. Since $Q$ is homotopy equivalent to a graph,
$H^2(Q,\Zo^n)=0$, so the Euler class of $Y$ vanishes. Thus in
origami case we have $Y=Q\times T^n$.

Now we restrict to the case when all proper faces of $Q$ are
acyclic. Since they are homotopy equivalent to graphs, it follows
automatically that they are contractible. Let $\ddb=\dim
H_1(Q)=\dim H_1(\Gamma)$. Poincare--Lefchetz duality implies:
\[
H_q(Q,\dd Q)\cong H^{n-q}(Q)\cong\begin{cases} \Zo,\mbox{ if } q=n;\\
\Zo^{\ddb},\mbox{ if }q=n-1;\\
0,\mbox{ otherwise.}
\end{cases}
\]
Let us describe the connecting homomorphisms $\delta_i\colon
H_i(Q,\dd Q)\to H_{i-1}(\dd Q)$. For simplicity we discuss the
case $n\geqslant 4$; dimensions $2$ and $3$ can be done similarly.
When $n\geqslant 4$, lacunas in the exact sequence of the pair
$(Q,\dd Q)$ imply that $\delta_i\colon H_i(Q,\dd Q)\to H_{i-1}(\dd
Q)$ is an isomorphism for $i=n-1,n$, and trivial otherwise. We
have
\[
H_i(\dd Q)\cong\begin{cases} \ko,\mbox{ if } i=0, n-1;\\
\ko^{\ddb},\mbox{ if } i=1, n-2;\\
0,\mbox{ otherwise.}
\end{cases}
\]

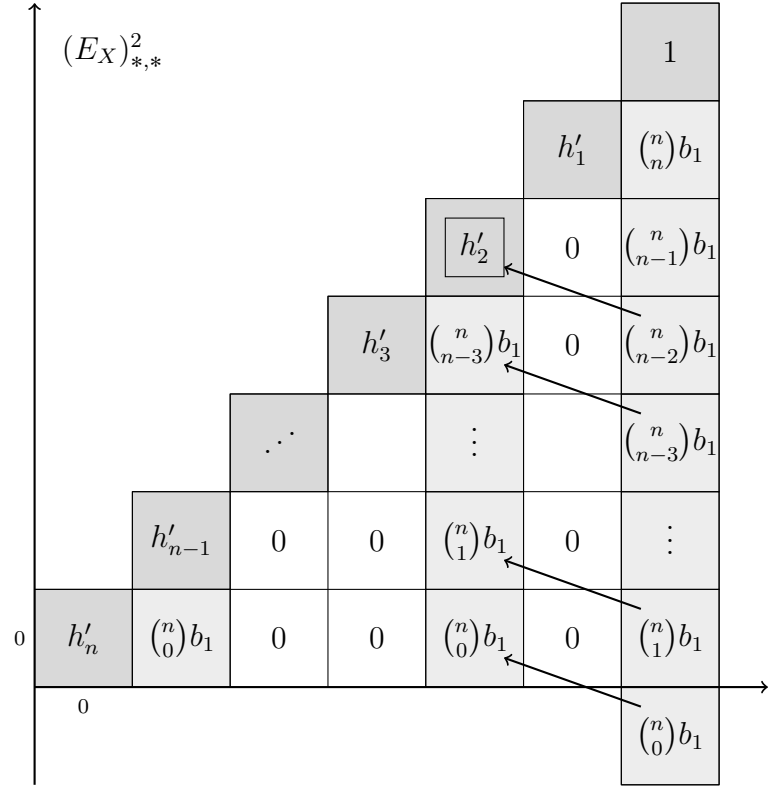
\begin{figure}[h]
\begin{center}
\begin{tikzpicture}[scale=1.3]

        \foreach \x in {0,1,...,6} \filldraw[fill=black!15,xshift=\x cm, yshift=\x cm] (0,0)--(0,1)--(1,1)--(1,0)--cycle;
        \foreach \x in {0,1,...,3} \filldraw[fill=black!7,yshift=\x cm] (4,0)--(4,1)--(5,1)--(5,0)--cycle;
        \foreach \x in {-1,0,...,5} \filldraw[fill=black!7,yshift=\x cm] (6,0)--(6,1)--(7,1)--(7,0)--cycle;
        \filldraw[fill=black!7] (1,0)--(1,1)--(2,1)--(2,0)--cycle;
        \draw (4.2,4.2)--(4.2,4.8)--(4.8,4.8)--(4.8,4.2)--cycle;

        \draw[->,thick]  (0,0)--(7.5,0); \draw[->,thick]  (0,-1)--(0,7);

        \draw (0,1)--(7,1); \draw (1,2)--(7,2); \draw
        (2,3)--(7,3); \draw (3,4)--(7,4); \draw (4,5)--(7,5); \draw
        (5,6)--(7,6); \draw (6,7)--(7,7); \draw (6,-1)--(7,-1);

        \draw (1,0)--(1,2); \draw (2,0)--(2,3); \draw
        (3,0)--(3,4); \draw (4,0)--(4,5); \draw (5,0)--(5,6); \draw
        (6,-1)--(6,7); \draw (7,-1)--(7,7);

        \draw (0.5,0.5) node{$h_n'$}; \draw (1.5,0.5) node{${n\choose
        0}\ddb$}; \draw (1.5,1.5) node{$h_{n-1}'$}; \draw (2.5,0.5)
        node{$0$}; \draw (2.5,1.5) node{$0$}; \draw (3.5,0.5)
        node{$0$}; \draw (3.5,1.5) node{$0$}; \draw (2.5,2.6)
        node{$\iddots$}; \draw (3.5,3.5) node{$h'_3$}; \draw (4.5,0.5) node{${n\choose
        0}\ddb$}; \draw (4.5,1.5) node{${n\choose 1}\ddb$}; \draw (4.5,2.6)
        node{$\vdots$}; \draw (4.5,3.5) node{${n\choose
        n-3}\ddb$}; \draw (4.5,4.5) node{$h_2'$}; \draw (5.5,0.5)
        node{$0$}; \draw (5.5,1.5) node{$0$}; \draw (5.5,3.5)
        node{$0$}; \draw (5.5,4.5) node{$0$}; \draw (5.5,5.5)
        node{$h_1'$}; \draw (6.5,-0.5) node{${n\choose 0}\ddb$}; \draw (6.5,0.5) node{${n\choose
        1}\ddb$}; \draw (6.5,1.6) node{$\vdots$}; \draw (6.5,2.5) node{${n\choose
        n-3}\ddb$}; \draw (6.5,3.5) node{${n\choose n-2}\ddb$}; \draw (6.5,4.5) node{${n\choose
        n-1}\ddb$}; \draw (6.5,5.5) node{${n\choose n}\ddb$}; \draw (6.5,6.5) node{$1$};

        \draw[->,thick] (6.2,3.8)--(4.8,4.3); \draw[->,thick]
        (6.2,2.8)--(4.8,3.3); \draw[->,thick]
        (6.2,0.8)--(4.8,1.3); \draw[->,thick] (6.2,-0.2)--(4.8,0.3);

        \draw (0.8,6.5) node{$\E{X}^2_{*,*}$}; \draw (0.5,-0.2) node{\tiny
        $0$}; \draw (-0.15,0.5) node{\tiny $0$};
\end{tikzpicture}
\end{center}
\caption{Spectral sequence for homology of orientable origami
manifold with acyclic proper faces of the orbit space. Instead of
entries we write their ranks.} \label{figOrigami}
\end{figure}

Proposition \ref{propEXstruct} implies that $\E{X}^2_{p,q}$ has
the form shown schematically on Figure~\ref{figOrigami}. The
differential $\dif{X}^2$ hitting the marked position produces
relations $R'_{\beta,A}$ of the second type on the cycles
$[X_I]\in H^{2n-4}(X)$. These relations are explicitly described
by Proposition \ref{propTwoTypesOfRels}, and the number of
independent relations is ${n\choose 2}\ddb$. Dually, this
consideration shows that the map $\ko[S]/\Theta\to H^*(X)$ has
nontrivial kernel of dimension ${n\choose 2}\ddb$ in degree $4$.

In addition, we have the following non-face cycles:
\begin{enumerate}
\item There are $\ddb$ one-dimensional spine classes, which appear
as the liftings of cycles in $\Gamma$.

\item There are $\ddb$ diaphragm classes of codimension $1$ given by
the generators of $H_{n-1}(Q,\dd Q)\cong H^1(\Gamma)$ swept by the
action of a whole torus. These cycles are equivariant. They are
dual to cycles (1).

\item There are $n\ddb$ extremal diaphragm classes of codimension
$2$. These are given by the generators of $H_{n-1}(Q,\dd Q)$ swept
around by actions of $(n-1)$-dimensional subtori. A choice of
these classes is not canonical.
\end{enumerate}

%
%
%
%
%
%
%

\section{Collar model}\label{SecCollar}

In this section we prove Theorem \ref{thmSocle} using an auxiliary
space $\wh{X}$.

\begin{con}
Consider the space $\wh{Q}=\dd Q\times [0,1]$. It is an
$(n-1)$-dimensional manifold with the boundary $\dd \wh{Q}$ of the
form $\dd_0\wh{Q}\sqcup\dd_1\wh{Q}$, where $\dd_\epsilon\wh{Q}=\dd
Q\times\{\epsilon\}$, $\epsilon=0,1$. We may identify $\dd_0
\wh{Q}$ with $\dd Q$ and consider $\wh{Q}$ as a filtered
topological space:
\[
Q_0\subset Q_1\subset\ldots\subset Q_{n-1}=\dd Q = \dd_0
\wh{Q}\subset \wh{Q}.
\]
One may think about $\wh{Q}$ as a collar of $\dd Q$ inside $Q$.

Consider the space $\wh{Y}=\wh{Q}\times T^n$ and the
identification space $\wh{X}=\wh{Y}/\simc$. The relation $\sim$
identifies points over $\dd_0\wh{Q}$ in the same way as it does
for $\dd Q\subset Q$, while there are no identifications over
$\dd_1\wh{Q}$. The space $\wh{X}$ is a manifold with boundary. Its
boundary consists of points over $\dd_1 Q$, hence is homeomorphic
to $\dd Q\times T^n$. The space $\wh{X}$ can be considered as a
$T^n$-invariant tubular neighborhood of the union of all
characteristic submanifolds in $X$. There are natural topological
filtrations on $\wh{Y}$ and $\wh{X}$ induced by the filtration on
$\wh{Q}$.
\end{con}

In terminology of \cite{AyV2} the space $\wh{X}$ is a Buchsbaum
pseudo-cell complex, thus Propositions \ref{propEQstruct},
\ref{propEYstruct}, and items (1)--(5) of Proposition
\ref{propEXstruct} hold for $\wh{Q}$, $\wh{Y}$, and $\wh{X}$. The
$n$-th column of all spectral sequences vanishes, since
$H_*(\wh{Q},\dd_0\wh{Q})=0$. Thus all spectral sequences
$\Ea{\wh{Q}}^r$, $\Ea{\wh{Y}}^r$, $\Ea{\wh{X}}^r$ collapse at a
first page and, consequently, the spectral sequences
$\E{\wh{Q}}^r$, $\E{\wh{Y}}^r$, $\E{\wh{X}}^r$ collapse at a
second page.

For each $I\in S_Q$, with $\dim F_I=q<n$ there is a distinguished
element $[X_I]\in H_{2q}(\wh{X}_q,\wh{X}_{q-1}) =
\E{\wh{X}}^1_{q,q}$. It survives in a spectral sequence, and gives
the fundamental class of the face submanifold $X_I\subset \wh{X}$.
Linear relations on classes $[X_I]$ in $H_*(\wh{X})$ are described
similarly to Section \ref{SecFaceSubmfds}. When $q=n-1$ there are
no relations on $[X_I]$, since there are no differentials landing
at the cell $\E{\wh{X}}^1_{n-1,n-1}$. For $q<n-1$ the relations on
$[X_I]$ are the images of $\dif{\wh{X}}^1\colon
\E{\wh{X}}^1_{q+1,q}\to\E{\wh{X}}^1_{q,q}$. These differentials
coincide with $\dif{X}^1$ thus the relations on $[X_I]$ for
$|I|>1$ are exactly $R_{J,A}$ defined in Proposition
\ref{propTwoTypesOfRels}. Thus Proposition \ref{propIsomQuotients}
implies

\begin{lemma}\label{lemmaCollarFaceClasses}
Let $V_*$ be a submodule of $H_*(\wh{X})$ generated by face
classes $[X_I]$, $I\neq\minel$. Then there is a degree reversing
linear map
\[
\tilde{\varphi}\colon V_{2q}\to (\ko[S_Q]/\Theta)_{2(n-q)}.
\]
sending $[X_I]$ to $v_I$. It is an isomorphism for $q<n-1$ and
surjective for $q=n-1$. This map is a ring homomorphism with
respect to the intersection product on $\wh{X}$.
\end{lemma}

Now we can give a geometrical proof of Theorem \ref{thmSocle}.

\begin{proof} We need to prove that the element $L'_{\beta,A}$
lies in a socle of $\ko[S_Q]/\Theta$. Thus we need to show that
$L'_{\beta,A}v_i=0$ for every vertex $i\in \ver(S_Q)$, where
$\beta\in H_q(\dd Q)$, $|A|=q$, and $q\leqslant n-2$. According to
Lemma \ref{lemmaCollarFaceClasses}, it is sufficient to show that
$R'_{\beta,A}\cap [F_i]=0$ in $H_*(\wh{X})$.

Consider a geometrical cycle $\mathcal{B}\subset \dd_0\wh{Q}$
representing $\beta$. Now we allow $\mathcal{B}$ be any
representative, and do not require that it lies in the stratum
$Q_q$. Consider the cycle $\mathcal{B}\times T^{(A)}$ in
$\dd_0\wh{Q}\times T^n$, and the corresponding cycle
$(\mathcal{B}\times T^{(A)})/\simc$ in $\wh{X}$. The latter cycle
represents the class $R'_{\beta,A}\in H_{2q}(\wh{X})$ by
definition. The space $\wh{Q}$ is homologous to $\dd_0\wh{Q}$ so
we can move the cycle $\mathcal{B}\subset \wh{Q}$ away from the
boundary $\dd_0\wh{Q}$. Thus $\mathcal{B}\cap [F_i]=\emptyset$ and
therefore $R'_{\beta,A}\cap [X_i]=0$ in $H_*(\wh{X})$.
\end{proof}

\begin{rem}
The same argument proves that $L'_{\beta,A}v_I=0$ for any simplex
$I\in S\setminus\minel$. This fact does not directly follow from
Theorem \ref{thmSocle}, since the map $\ko[m]\to\ko[S]/\Theta$ may
be not surjective in general.
\end{rem}

\begin{rem}
The only reason why we considered the collar model $\wh{X}$
instead of $X$ is that there are no additional relations in
$H_*(\wh{X})$ compared with $\ko[S_Q]/\Theta$. The space $\wh{X}$
captures the properties of $\ko[S_Q]/\Theta$ more precisely than
$X$. On the other hand, $\wh{X}$ is a manifold with boundary so
there is a geometrical intersection theory on it. This makes it an
object worth studying.
\end{rem}

\begin{rem}
The classes $R'_{\beta,A}\in \E{X}^2_{q,q}$ are the images of the
classes $\beta\times [T^{(A)}]\in H_q(\dd Q)\times H_q(T^n)$ under
the homomorphism $\fa_*^1\colon
\Ea{\wh{Y}}^1_{q,q}\to\Ea{\wh{X}}^1_{q,q}$. This homomorphism is
injective by Proposition \ref{propEXstruct}, thus the construction
gives an inclusion
\[
H_q(\dd Q)\otimes H_q(T^n)\hookrightarrow
\soc(\ko[S_Q]/\Theta)_{2(n-q)}
\]
for each $q\leqslant n-2$. When $q=n-1$, the map $H_{n-1}(\dd
Q)\otimes H_{n-1}(T^n)\to \soc(\ko[S_Q]/\Theta)_{2}$ has the
kernel of the form $\langle[\dd Q]\rangle\otimes H_{n-1}(T^n)$,
where $[\dd Q]$ denotes the fundamental class of $\dd Q$. Note
that $\dd Q$ may be disconnected, thus there could exist classes
in $H_{n-1}(\dd Q)$ different from $[\dd Q]$. So far, there exists
an injective map
\[
(H_{n-1}(\dd Q)/\langle[\dd Q]\rangle)\otimes
H_{n-1}(T^n)\hookrightarrow \soc(\ko[S_Q]/\Theta)_{2}
\]
These statements reprove the result of Novik--Swartz
\cite[Th.3.5]{NS} in case of homology manifolds.
\end{rem}

\end{document}